\PassOptionsToPackage{dvipsnames}{xcolor}
\documentclass[11pt,reqno,oneside]{amsart}
\usepackage{amsmath, amssymb, graphicx, mathrsfs, dsfont, hyperref}
\usepackage{epsfig}
\usepackage{setspace}
\usepackage{amsthm}
\usepackage[top=1in, bottom=1in, left=1in, right=1in]{geometry} 
\usepackage{pgfplots}
\usepackage[dvipsnames]{xcolor}
\usepackage{mathtools}
\usepackage{skull, color}
\usepackage{mathabx}
\usepackage{float}
\usepackage{cite}
\usepackage{caption}

\theoremstyle{definition}
\newtheorem{theorem}{Theorem}[section]
\newtheorem{lemma}[theorem]{Lemma}
\newtheorem{corollary}[theorem]{Corollary}

\newtheorem{conjecture}[theorem]{Conjecture}
\newtheorem{proposition}[theorem]{Proposition}
\newtheorem{definition}[theorem]{Definition}
\newtheorem{example}[]{Example}
\newtheorem*{theorem*}{Theorem}
\newtheorem*{conjecture*}{Conjecture}
\newtheorem*{question*}{Question}
\newtheorem*{lemma*}{Lemma}
\newtheorem*{corollary*}{Corollary}
\newtheorem*{definition*}{Definition}
\newtheorem*{example*}{Example}

\newcommand{\ZZ}{\mathbb{Z}}     % Integers
\newcommand{\NN}{\mathbb{N}}     % Naturals
\title{The Fibonacci Sequence is Normal Base 10}

\author[Brennan Benfield]{Brennan Benfield}
\address{Brennan Benfield: Department of Mathematics, University of Hawaii, 2565 McCarthy Mall, Honolulu, HI 96822, USA}
\email{Benfield@math.hawaii.edu}

\author[Michelle Manes]{Michelle Manes}
\thanks{Second author partially supported by Simons Foundation grant number 359721.}
\address{Michelle Manes: Department of Mathematics, University of Hawaii, 2565 McCarthy Mall, Honolulu, HI 96822, USA}
\email{mmanes@math.hawaii.edu}


\begin{document}

\maketitle
\begin{abstract}
In this paper, we show that the concatenation of the Fibonacci sequence is \textit{normal} in base $10$, meaning every string of a given length, $k$, occurs as frequently as every other string of length $k$ (there are as many $1$'s as $2$'s and as many $704$'s and $808$'s). Although we know that almost every number is normal, we can name very few of them. It is still unclear if $e$, $\pi$, or $\sqrt{2}$ are normal. We show that concatenating the Fibonacci sequence behind a decimal creates a normal number in every base of the form $5^x\times2^y$. We then provide evidence that potentially extends our result to all integer bases, and claim that the Fibonacci concatenation is \textit{absolutely normal}.
\end{abstract}
\section{Introduction}

The sequence of Fibonacci numbers $\left(F_{n}\right)_{n\in\NN}$ is given by the recurrence equation: 
\begin{align*}
F_0&=0\\ 
F_1&=1\\ 
F_n&=F_{n-1}+F_{n-2}\\
\left(F_n\right)_{n\in\NN}&=0,1,1,2,3,5,8,13,21,34,55,89,144,233,377,610,987,1597,2584,4181,6765,10946\ldots
\end{align*}

In this paper, we investigate whether this famous sequence, when concatenated behind a decimal, creates a \textit{normal number}. A number in a base $\beta$ is $\textit{normal}$ if every string of length $k$ occurs in the limit as often as every other string of length $k$, namely with a frequency of $1/\beta^k$. To prove that the Fibonacci concatenation is normal in a base $\beta$, we use a technique that measures the frequency of each digit in every place value of the Fibonacci numbers in base $\beta$. A major tool we use is the well-known fact  that the Fibonacci sequence is purely periodic modulo every integer base $\beta$. The length of one period of the Fibonacci sequence modulo a base $\beta$ is known as the \textit{Pisano period}. The aim of this paper is to connect the notion of normal numbers and Pisano periods, to prove that the Fibonacci concatenation is normal in Base $10$, and to provide computations and heuristic evidence that the Fibonacci concatenation \textit{absolutely normal}. 
\begin{theorem}[Main Theorem]
The Fibonacci sequence, when concatenated behind a decimal, is normal in every base of the form $5^x2^y$ for nonnegative integers $x$ and $y$.
\end{theorem}

For the convenience of the reader, we provide here an outline of the paper and a summary of techniques used:  
\begin{description}
\item [Section 2] Background information on normal numbers and Pisano periods.
\item [Section 3] Connecting the relationship between periodicity of the Fibonacci sequence modulo an integer $\beta$ with the sequences of digits in each place value of the Fibonacci numbers in base $\beta$.
\item [Section 4] Proof that the Fibonacci concatenation is normal in every base of the form $\beta=5^x$.
\item [Section 5] Proof that the Fibonacci concatenation is normal in every base of the form $\beta=2^y$.
\item [Section 6] Combining the results from the previous two sections to show that the Fibonacci concatenation is normal in every base of the form $\beta=5^x2^y$.
\item [Section 7] Computational evidence that the Fibonacci concatenation is in fact normal in every base, thus is \textit{absolutely normal}.
\end{description}

\section{Preliminaries}
There are many definitions of a normal number. The first was given by Borel in 1909 \cite{[Borel]}, but the definition has been refined during the last century. We use the definition given by Davenport and Erdös \cite{[Davenport]}. 
\begin{definition}
Let $x$ be a real number in a base $\beta$. Let $a_1a_2\cdots a_k$ be any string of length $k$ written in base $\beta$. Let $N(t)$ denote the number of times that this string occurs among the first $t$ digits of $x$. Then, $x$ is \textit{normal} in $\beta$ if
\[\lim_{t\to\infty}\frac{N(t)}{t}=\frac{1}{\beta^k}.\]
In the case where $k=1$, we say $x$ is \textit{simply normal} in base $\beta$ if
\[\lim_{t\to\infty}\frac{N(t)}{t}=\frac{1}{\beta}.\]
Finally, a number is \textit{absolutely normal} if it is normal in every base.
\end{definition}

More plainly, a number is \textit{normal} in a base $\beta$ if every string of length $k$ occurs with a limiting relative frequency of $1/\beta^k$ and is \textit{simply normal} if each digit occurs with a limiting relative frequency of $1/\beta$. The equivalent definition of \textit{normal} that was given by Borel \cite{[Borel]} in 1909 was further refined by Pillai \cite{[Pillai]} in 1940:

\begin{lemma}[Pillai]\label{lemma:Pillai}
A number is \textit{normal} in base $\beta$ if and only if it is simply normal in base $\beta^k$ for all positive integers $k$.
\end{lemma}

Borel \cite{[Borel]}  showed that in every base, almost every number is normal, in the sense that non-normal numbers have Lebesgue measure $0$. Despite their overwhelming population on the real line, mathematicians have yet to prove that any of the commonplace mathematical constants like $\pi$, $e$, $\sqrt{2}$, $\phi$ etc. are normal. Because of this, the earliest examples of normal numbers were  concatenations of well known sequences.  Champernowne \cite{[Champernowne]} showed the concatenation of the natural numbers is normal in base $10$. Other early examples of normal numbers in base $10$ include the concatenation of the square numbers (Besicovitch \cite{[Besicovitch]}) and the concatenation of the primes (Copeland and Erd\H{o}s \cite{[Copeland]}).
\begin{align}
\nonumber &\text{Champernowne}&.123456789101112131415161718192021222324252\ldots\\
\nonumber &\text{Besicovitch}&.149162536496481100121144169196225256289324\ldots\\
\nonumber &\text{Copeland \& Erd\H{o}s}&.235711131719232931374143475359616771737983\ldots
\end{align}

We consider here the concatenation of the Fibonacci sequence  behind a decimal. Let $\left(F_{n,\beta}\right)$ denote the Fibonacci sequence in base $\beta$ and let $\left<\left(F_{n,\beta}\right)\right>$ denote its concatenation behind a decimal. 
\begin{example*}
In bases $2,\ldots,10$, the concatenation of the Fibonacci sequence is given by the following:
\begin{align}
\nonumber \left<\left(F_{n,2}\right)\right>=&.01110111011000110110101100010110111101100110010000111\ldots\\
\nonumber \left<\left(F_{n,3}\right)\right>=&.01121012221112101021200110022121002212211122221112111\ldots\\
\nonumber \left<\left(F_{n,4}\right)\right>=&.01123112031111202313112121003221113212120233123120331\ldots\\
\nonumber \left<\left(F_{n,5}\right)\right>=&.01123101323411142103241034141330024420124222234240314\ldots\\
\nonumber \left<\left(F_{n,6}\right)\right>=&.01123512213354131225400102514252454432311221155443120\ldots\\
\nonumber \left<\left(F_{n,7}\right)\right>=&.01123511163046106155264452104615312610444110351151222\ldots\\
\nonumber \left<\left(F_{n,8}\right)\right>=&.01123510152542671312203515711142173330755030101251515\ldots\\
\nonumber \left<\left(F_{n,9}\right)\right>=&.01123581423376110817027845874713162164348156551024616\ldots\\
\nonumber \left<\left(F_{n,10}\right)\right>=&.01123581321345589144233377610987159725844181676510946\ldots
\end{align}
\end{example*}
There are periodic patterns in the Fibonacci sequence for every integer modulus. In 1877, Lagrange \cite{[Lagrange]} noted that the Fibonacci sequence modulo $10$ repeats every $60$ terms. Several specific results for the periodicity of the Fibonacci sequence modulo various integers followed, culminating in the following theorem of  Wall from 1960~\cite{[Wall]}.

\begin{theorem}[Wall]
For every positive integer $m$, the Fibonacci sequence is periodic modulo $m$.
\end{theorem}

The length of this period is known as the \textit{Pisano period} and is denoted $\pi(m)$. Pisano periods have been generalized for Lucas numbers, Pell numbers, $(a,b)$-Fibonacci numbers, and $n$-step Fibonacci numbers. For our purpose, we are only concerned with the Fibonacci sequence, and so will use the following definition.
\begin{definition}
The \textit{Pisano period} of a positive integer $m$, denoted $\pi(m)$, is the smallest number $k>0$ such that $F_k\equiv0\mod{m}$ and $F_{k+1}\equiv1\mod{m}$.
\[\text{-Equivalently-}\]
The Pisano period $\pi(m)$ is the length of one (shortest) period of the Fibonacci sequence modulo~$m$.
\end{definition}
\begin{example*} The sequence of Pisano periods can be found at OEIS: A001175 \cite{[A001175]}. Below are the first few Pisano periods up to $20$. 
\begin{table}[H]
\centering
\begin{tabular}{c|c|c|c|c|c|c|c|c|c|c|c|c|c|c|c|c|c|c|c}
$m$&2&3&4&5&6&7&8&9&10&11&12&13&14&15&16&17&18&19&20\\
\hline
$\pi(m)$&3&8&6&20&24&16&12&24&60&10&24&28&48&40&24&36&24&18&60\\
\end{tabular}
\caption{Pisano period of the first $20$ integers.}
\end{table}
\end{example*}
Notice that the only odd integer that is a Pisano period is $\pi(2)=3$. Combining this with the results of Stanley \cite{[Stanley]} and Ehrlich \cite{[Ehrlich]}, Renault \cite{[Renault]} showed the following:
\begin{proposition}[Renault]
For every even integer $n>4$, there is an integer $m$ such that $\pi(m)=n$.
\end{proposition}

A property that we will use frequently is the \textit{nesting property} of Pisano periods. For every positive integer $n$, the Pisano period of $n$ divides the Pisano period of every multiple of $n$.
\begin{lemma}[Wall \cite{[Wall]}]\label{nesting}
 For any two positive integers $m,n$, if $m\mid n$ then $\pi(m)\mid\pi(n)$.
\end{lemma}
For our purpose, we only need the nesting property for powers of the base we are considering. Given an integer $m$,  $\pi(\beta^i)\mid\pi(\beta^{i+1})$. 
\begin{example}
Following the investigation of Legrange, Geller \cite{[Geller]}  observes the nesting of powers of 10:

\[
\pi(10)=60, \qquad \pi(100)=300, \qquad \pi(1000)=1500, \qquad \pi(10000)=15000, \ldots 
\]
\end{example}
This pattern was later generalized by Jarden~\cite{[Jarden]} into the following proposition:

\begin{proposition}[Jarden]
The Fibonacci sequence is periodic modulo powers of ten with the following periods:
\[\pi(10^k)=\begin{cases}
60&k=1\\
300&k=2\\
15\cdot 10^{k-1}&k\vargeq3\qquad .
\end{cases}\]
\end{proposition}

From the Lemma~\ref{nesting} it follows that for a positive prime $p$, $\pi(p^i)\varleq\pi(p^{i+1})$. We can refine this expression to strict inequality, provided the prime $p$ is not a Wall-Sun-Sun prime.

\begin{definition}[Wall, Sun, Sun]
A Wall-Sun-Sun prime is a prime number $p$ such that $\pi(p)=\pi(p^2)$.
\end{definition}
Although there are conjectured to be infinitely many Wall-Sun-Sun primes \cite{[Klaska]}, none have been found. If it can be shown that Wall-Sun-Sun primes do not exist, that would imply Fermat's last theorem \cite{[Sun-Sun]}. Wall \cite{[Wall]} proposed the following lemma that tests whether a candidate integer can be a Wall-Sun-Sun prime. 
\begin{lemma}[Wall]\label{lemma:WallCriteria}
Given a prime $p$, let $t$ be the largest integer such that $\pi(p^t)=\pi(p)$. Then for all $e\vargeq t$, $\pi(p^e)=p^{e-t}\pi(p)$.
\end{lemma}

The online research community \underline{PrimeGrid} is currently leading a computer search for Wall-Sun-Sun primes with its leading edge over $3\times10^{17}$. Because we have still not found a Wall-Sun-Sun prime, we can say that for every known prime that has been tested, $t=1$ and $\pi(p^e)=p^{e-1}\pi(p)$.
\begin{corollary}\label{conj:PiGrowth1}
For every prime $p$ that is not a Wall-Sun-Sun prime, $\pi(p^i)<\pi(p^{i+1})$ for $i \in \ZZ^+$.
\end{corollary}

\begin{proof}
Because we are supposing that $p$ is not a Wall-Sun-Sun prime, Wall's lemma \ref{lemma:WallCriteria} tells us that for all $e\vargeq1$, $\pi(p^e)=p^{e-1}\pi(p)$. Then $\pi(p^i)=p^{i-1}\pi(p)<p^{(i+1)-1}\pi(p)=\pi(p^{i+1})$.
\end{proof}

Leveraging the following Lemma by Wall \cite{[Wall]}, we can extend Corollary \ref{conj:PiGrowth1} to composite integers.
\begin{lemma}[Wall]\label{lemma:lcm}
If $m$ has the prime factorization $m=p_1^{e_1}p_2^{e_2}\cdots p_n^{e_n}$, then
\[
\pi(m)=\text{LCM}\left(\pi(p_1^{e_1}),\pi(p_2^{e_2}),\cdots ,\pi(p_n^{e_n})\right).
\]
\end{lemma}
Then we have the following result:
\begin{proposition}\label{conj:PiGrowth2}
For every integer $m$ that has no Wall-Sun-Sun primes in its prime factorization, $\pi(m^i)<\pi(m^{i+1})$ for $i \in \ZZ^+$.
\end{proposition}

\begin{proof}
Suppose $m$ has the prime factorization $m=p_1^{e_1}p_2^{e_2}\ldots p_n^{e_n}$, where no $p$ is a Wall-Sun-Sun prime. Then we have the following:
\[
\pi(m^i)=\text{LCM}\left(\pi(p_1^{{ie_1}})\pi(p_2^{{ie_2}})\cdots\pi(p_n^{{ie_n}})\right).
\]
And by Theorem \ref{conj:PiGrowth1} we have that for all $j=1,2,\ldots,n$, $\pi(p_j^{ie_j})<\pi(p_j^{(i+1)e_j})$. Then
\begin{align*}
\pi(m^i)&=\text{LCM}\left(\pi(p_1^{{ie_1}})\pi(p_2^{{ie_2}})\cdots\pi(p_n^{i(e_n)})\right)\\
&<\text{LCM}\left(\pi(p_1^{{(i+1)e_1}})\pi(p_2^{{(i+1)e_2}})\cdots\pi(p_n^{{(i+1)e_n}})\right)\\
&=\pi(m^{i+1}).\qedhere
\end{align*}
\end{proof}

\section{Bases and Moluli}
We prove the Main Theorem by exploiting the connection between the digits in the $\beta^k$'s place of the Fibonacci sequence in base $\beta$ and the Fibonacci sequence modulo $\beta^{k+1}$. To extract the digits $0,1,2,\ldots,\beta-1$ in the $\beta^k$'s place, we introduce the following function:
\begin{definition}
Denote by $\left(\Phi_{\beta^k}(n)\right)_{n\in\mathbb{N}}$ the sequence of $\beta^k$'s place digits in the Fibonacci sequence in base $\beta$. To obtain this sequence, we can use the following function:
\begin{align*}
\Phi_{\beta^k}:\mathbb{N}&\rightarrow\mathbb{Z}/\beta\mathbb{Z}  \\
n&\mapsto\left\lfloor \frac{F_n}{\beta^k} \right\rfloor \mod{\beta} .
\end{align*}

\end{definition}

Let $\left(\Phi_{\beta^k}(n)\right)_{n=1}^{N_k}$ denote one (shortest) period of the $\beta^k$'s place digits of the Fibonacci sequence in base $\beta$ where $N_k$ is the length of a shortest period of the sequence $\left(\Phi_{\beta^k}(n)\right)_{n\in\mathbb{N}}$.  This notation is justified in the following Lemma.
\begin{lemma}
\label{thm:PiDigits}
The sequence $\left(\Phi_{\beta^k}(n)\right)_{n\in\mathbb{N}}$ is periodic. Furthermore, if the base $\beta$ is not divisible by any Wall-Sun-Sun primes, the length $N_k$ of its period is equal to  $\pi(\beta^{k+1})$.
\end{lemma}

\begin{proof}
Periodicity of the sequence $\left(\Phi_{\beta^k}(n)\right)_{n\in\mathbb{N}}$  follows immediately from the periodicity of the Fibonacci sequence modulo $\beta^{k+1}$; all of the digits must repeat after $\pi(\beta^{k+1})$ terms, so the $k^\text{th}$ digit (in base $\beta$) certainly repeats every $\pi(\beta^{k+1})$ terms. It follows that $N_k \mid \pi(\beta^{k+1})$.

To show equality, we proceed by contradiction.  Suppose  that we do not have equality, then $N_k < \pi(\beta^{k+1})$. That says that the $k^\text{th}$ digit (in base $\beta$) repeats on a shorter period than the Fibonacci number modulo $\beta^{k+1}$. So the sequence modulo $\beta^{k+1}$, ignoring the $k^\text{th}$ digit, would have the same period. In other words,  $\pi(\beta^k ) = \pi(\beta^{k+1})$. But this contradicts Proposition~\ref{conj:PiGrowth2}.
\end{proof}

\begin{example*}
Consider the Fibonacci sequence in ternary (base $3$). The $3^0$'s place digits are given by the sequence:
\begin{align}
\nonumber \left(\Phi_{3^0}(n)\right)_{n=1}^{8}=0, 1, 1, 2, 0, 2, 2, 1\circlearrowleft .
\end{align}
This sequence repeats every $8$ terms. The length exactly coincides with the Pisano period $\pi(3)=\pi(\beta^{1})$ (this sequence is indeed the Fibonacci sequence modulo $3$). Consider now the sequence of $3^1$'s place digits in base $3$:
\begin{align}
\nonumber \left(\Phi_{3^1}(n)\right)_{n=1}^{24}=0, 0, 0, 0, 1, 1, 2, 1, 1, 2, 0, 2, 0, 2, 2, 2, 2, 1, 0, 1, 2, 0, 2, 0\circlearrowleft .
\end{align}
This sequence repeats every $24$ terms. The length exactly coincides with the Pisano period $\pi(9)=\pi(\beta^{2})$. Again, consider the sequence of $3^2$'s place digits in base $3$:
\begin{align}
\nonumber \left(\Phi_{3^2}(n)\right)_{n=1}^{72}=&0, 0, 0, 0, 0, 0, 0, 1, 2, 0, 0, 0, 1, 1, 2, 1, 1, 0, 2, 2, 1, 1, 2, 1, \\
\nonumber &1, 2, 0, 2, 2, 1, 0, 2, 0, 2, 0, 2, 0, 2, 2, 2, 2, 2, 2, 1, 0, 2, 2, 2, \\
\nonumber &2, 1, 0, 1, 1, 2, 0, 0, 1, 1, 0, 1, 2, 0, 2, 0, 0, 1, 2, 0, 2, 0, 2, 0\circlearrowleft .
\end{align}
This sequence repeats every $72$ terms. The length exactly coincides with the Pisano period of $\pi(27)=\pi(\beta^3)$.
\end{example*}

It is well known \cite{[Gupta], [Renault]} that the integer bases, with respect to the Fibonacci sequence, can be evenly divided into three categories according to the number of zeros in their Pisano period.
\begin{definition}
For a base $\beta$, denote by $\omega(\beta)$ the number of zeros in one Pisano period of $\beta$.
\end{definition} 
The sequence produced by $\omega(n)$ for $n=1,2,3,\ldots$ can be found at OEIS A001176 \cite{[A001176]}. 
\begin{table}[H]
\centering
\begin{tabular}{c|c|c|c|c|c|c|c|c|c|c|c|c|c|c|c|c|c|c|c}
$\beta$&2&3&4&5&6&7&8&9&10&11&12&13&14&15&16&17&18&19&20\\
\hline
$\pi(\beta)$&3&8&6&20&24&16&12&24&60&10&24&28&48&40&24&36&24&18&60\\
\hline
$\omega(\beta)$&1&2&1&4&2&2&2&2&4&1&2&4&2&2&2&4&2&1&2
\end{tabular}
\caption{Table of $\beta$, $\pi(\beta)$, and $\omega(\beta)$  for $2\varleq\beta\varleq20$.}
\end{table}
It has been proven by Gupta et.~al.~\cite{[Gupta]} that the only possible values for $\omega(n)$ are $1, 2,$ and $4$. In the table below, we categorize the integer bases according to the number of zeros in their Pisano period.

\begin{table}[H]
\centering
\begin{tabular}{c|c|c}
&&\\
$\omega(\beta)=1$&$\omega(\beta)=2$&$\omega(\beta)=4$\\
&&\\
\hline
1, 2, 4, 11, 19, 22, 29, 31,&3, 6, 7, 8, 9, 12, 14, 15,&5, 10, 13, 17, 25, 26, 34, 37,\\
38, 44, 58, 59, 62, 71, 76, 79,&16, 18, 20, 21, 23, 24, 27, 28,&50, 53, 61, 65, 73, 74, 85, 89,\\
101, 116, 118, 121, 124, 131,&30, 32, 33, 35, 36, 39, 40,&97, 106, 109, 113, 122, 125,\\
 139,142, 151, 158, 179, 181,\ldots&41, 42, 43, 45, 46, 47, 48, 49,\ldots&130, 137, 146, 149, 157\ldots\\
&&\\
OEIS: A053031 \cite{[A053031]}&OEIS: A053030 \cite{[A053030]}&OEIS: A053029 \cite{[A053029]}
\end{tabular}
\caption{Splitting of bases $\beta$ according to the number of zeros in one Pisano period}
\end{table}
Although there are countably many integer bases in each category, the rate at which each category grows is not balanced. Of the first ten thousand bases, $1013$ bases have one zero in their Pisano period, $7917$ bases have two zeros in their Pisano period, and $1070$ bases have four zeros in their Pisano period. Evidence of this uneven distribution in found in a series of conjectures on the OEIS: A053029, A053030, and A053031 \cite{[A053029], [A053030], [A053031]} concerning the number of zeros in an integer's Pisano period. We have summarized the three conjectures from the OEIS into the following:
\begin{conjecture}\label{conj:oeis}[OEIS: A053029, A053030, A053031]\\
\\
\romannumeral1.\relax \ An integer $m$ has four zeros in its Pisano period if and only if $m$ is an odd number, all of whose factors have four zeros in their Pisano period, or if $m$ is twice such a number.\\
\\
\romannumeral2.\relax \ An integer $m$ has one zero in its Pisano period if and only if $m$ is an odd number, all of whose factors have one zero in their Pisano period, or if $m$ is twice or four times such a number.\\
\\
\romannumeral3.\relax \ Every other integer has one zero in its Pisano period.
\end{conjecture}

There is a table of relations compiled by Vinson \cite{[Vinson]} that was distilled into a theorem of Renault \cite{[Renault]} categorizing when a number has exactly $1$, $2$, or $4$ zeros in its Pisano period. 
\begin{theorem}[Renault]\label{thm:Renault}
Let $m$ and $n$ be integer bases, then:
\begin{table}[H]
\centering
\begin{tabular}{cc|ccc}
&&&$\omega(m)$&\\
&&&&\\
&&1&2&4\\
\hline
&&&&\\
&1&1&2&4 if $m=2$, else 2\\
&&&&\\
$\omega(n)$&2&2&2&2\\
&&&&\\
&4&4 if $n=2$, else 2&2&4
\end{tabular}
\caption{Table of $\omega\left(\text{LCM}(m,n)\right)$.}
\end{table}
\end{theorem}
We will primarily concern ourselves in this paper with bases $\beta$ where $\omega(\beta)=4$. There is a stability to these bases that allows us to determine if the Fibonacci concatenation is normal. First we will show that in the smallest such base, base $5$, the Fibonacci concatenation is \textit{normal}.

\section{Bases of the form $5^x$}
The technique to prove the Main Theorem relies on measuring the distribution of the digits of the Fibonacci sequence. We are particularly interested in the case where the digits are \textit{uniformly distributed}.
\begin{definition}
Given a number in base $\beta$, the digits $0,1,2,\ldots,\beta-1$ are \textit{uniformly distributed} if every digit occurs with a frequency of $1/\beta$.
\end{definition}

Consider the Fibonacci sequence in base $5$. The digits in base $5$ are $0,1,2,3,4$. Then the Fibonacci sequence in base 5 begins:
\begin{align}
\nonumber (F_{n,5})=0, 1, 1, 2, 3, 10, 13, 23, 41, 114, 210, 324, 1034, 1413, 3002, 4420, 12422, 22342, 40314, \ldots
\end{align}
There is a curious fact \cite{[Niederreiter]} about the distribution of the digits in the Fibonacci sequence that only holds modulo a power of $5$:
\begin{lemma}[Niederreiter \cite{[Niederreiter]}]\label{lemma:Niederreiter}
The digits $0,1,2,\ldots,5^k-1$ in the Fibonacci sequence modulo $5^k$ are uniformly distributed.
\end{lemma}

\begin{proposition}\label{prop:base5unif}
For every positive integer $k$, the digits $0,1,2,3,4$ are uniformly distributed in the sequences $\left(\Phi_{5^k}(n)\right)_{n=1}^{N_k}$.
\end{proposition}
\begin{proof}
We proceed by induction on $k$. 
For $k=0$, Lemma~\ref{lemma:Niederreiter} says that each digit is uniformly distributed in $\left(\Phi_{5^0}(n)\right)_{n=1}^{N_0}$. Now, suppose for all $0\varleq i< j$ that the sequences $\left(\Phi_{5^i}(n)\right)_{n=1}^{N_i}$ are uniformly distributed and consider the sequence $\left(\Phi_{5^{j}}(n)\right)_{n=1}^{N_{j}}$. By Lemma~\ref{lemma:Niederreiter}, we know that the Fibonacci sequence modulo $5^{j+1}$ is uniformly distributed. Further, the digits in each place $5^i$ for all $0\varleq i<j$   are uniformly distributed. So the digits of the Fibonacci sequence in the $5^{j}$'s place must be uniformly distributed. It follows that the digits in the sequence $\left(\Phi_{5^{j}}(n)\right)_{n=1}^{N_{j}}$ are uniformly distributed.
\end{proof}

\begin{example*}
Consider the function $\Phi_{5^0}(n)$. A single period of the $5^0$'s place is given below:
\begin{align}
\nonumber \left(\Phi_{5^0}(n)\right)_{n=1}^{20}=0,1,1,2,3,0,3,3,1,4,0,4,4,3,2,0,2,2,4,1\circlearrowleft .
\end{align}
Notice that each digit $0,1,2,3,4$ appears \textit{exactly} four times. After each period of twenty terms of the Fibonacci sequence in base $5$, each digit in the $5^0$'s place has a frequency of $1/5$; the digits are uniformly distributed. Next, consider the $5^1$'s place. The length of the period is given by $\pi(5^2)=100$. A single period of the $5^1$'s place is given below:
\begin{align}
\nonumber \left(\Phi_{5^1}(n)\right)_{n=1}^{100}=&0,0,0,0,0,1,1,2,4,1,1,2,3,1,0,2,2,4,1,1,3,4,2,1,3,\\
\nonumber &0,3,3,2,0,3,3,1,0,2,3,0,3,3,2,1,3,4,2,1,4,0,4,0,4,\\
\nonumber &0,4,4,4,4,4,3,2,0,3,4,2,1,3,4,3,2,0,3,3,2,0,2,3,1,\\
\nonumber &0,1,1,2,4,2,1,3,4,2,2,4,1,1,2,4,1,0,2,3,1,4,0,4,0\circlearrowleft .
\end{align}
Notice again that each digit $0,1,2,3,4$ occurs \textit{exactly} twenty times; the digits are uniformly distributed. Further, there are exactly five periods of $5^0$'s place digits nested in every one period of $5^1$'s place digits. We can continue this pattern for every place value of the Fibonacci sequence in base $5$.
\end{example*}
\begin{theorem}\label{thm:coloring}
The concatenation of the Fibonacci sequence in base $5$ is \textit{simply normal}.
\end{theorem}
\begin{proof}
Because the digits in every place value of the Fibonacci sequence in base $5$ are uniformly distributed, we can construct its concatenation in a particular way that demonstrates that it is simply normal. Suppose we color the digits Fibonacci sequence so that every digit that occurs in the $k$'s place value is colored the same: 
\begin{align}
\nonumber {\color{BrickRed}{\left(\Phi_{5^0}(n)\right)_{n\in\mathbb{N}}}}\ \ \ \ \ {\color{NavyBlue}{\left(\Phi_{5^1}(n)\right)_{n\in\mathbb{N}}}}\ \ \ \ \ {\color{OliveGreen}\left(\Phi_{5^2}(n)\right)_{n\in\mathbb{N}}}\ \ \ \ \ {\color{Plum}{\left(\Phi_{5^3}(n)\right)_{n\in\mathbb{N}}}}\ \ \ \ \ {\color{Dandelion}{\left(\Phi_{5^4}(n)\right)_{n\in\mathbb{N}}}}
\end{align}
\begin{align}
\nonumber (F_{n,5})=	{\color{BrickRed}{0}}, {\color{BrickRed}{1}}, {\color{BrickRed}{1}}, {\color{BrickRed}{2}}, {\color{BrickRed}{3}}, {\color{NavyBlue}{1}}{\color{BrickRed}{0}}, {\color{NavyBlue}{1}}{\color{BrickRed}{3}}, {\color{NavyBlue}{2}}{\color{BrickRed}{3}}, {\color{NavyBlue}{4}}{\color{BrickRed}{1}}, {\color{OliveGreen}{1}}{\color{NavyBlue}{1}}{\color{BrickRed}{4}}, {\color{OliveGreen}{2}}{\color{NavyBlue}{1}}{\color{BrickRed}{0}}, {\color{OliveGreen}{3}}{\color{NavyBlue}{2}}{\color{BrickRed}{4}}, {\color{Plum}{1}}{\color{OliveGreen}{0}}{\color{NavyBlue}{3}}{\color{BrickRed}{4}}, {\color{Plum}{1}}{\color{OliveGreen}{4}}{\color{NavyBlue}{1}}{\color{BrickRed}{3}}, {\color{Plum}{3}}{\color{OliveGreen}{0}}{\color{NavyBlue}{0}}{\color{BrickRed}{2}}, {\color{Plum}{4}}{\color{OliveGreen}{4}}{\color{NavyBlue}{2}}{\color{BrickRed}{0}}, {\color{Dandelion}{1}}{\color{Plum}{2}}{\color{OliveGreen}{4}}{\color{NavyBlue}{2}}{\color{BrickRed}{2}}, {\color{Dandelion}{2}}{\color{Plum}{2}}{\color{OliveGreen}{3}}{\color{NavyBlue}{4}}{\color{BrickRed}{2}}, {\color{Dandelion}{4}}{\color{Plum}{0}}{\color{OliveGreen}{3}}{\color{NavyBlue}{1}}{\color{BrickRed}{4}}, \ldots
\end{align}
Then the concatenation of the Fibonacci sequence in base 5 is given by:
\begin{align}
\nonumber \left<(F_{n,5})\right>=	{\color{BrickRed}{0}}{\color{BrickRed}{1}}{\color{BrickRed}{1}}{\color{BrickRed}{2}}{\color{BrickRed}{3}}{\color{NavyBlue}{1}}{\color{BrickRed}{0}}{\color{NavyBlue}{1}}{\color{BrickRed}{3}}{\color{NavyBlue}{2}}{\color{BrickRed}{3}}{\color{NavyBlue}{4}}{\color{BrickRed}{1}}{\color{OliveGreen}{1}}{\color{NavyBlue}{1}}{\color{BrickRed}{4}}{\color{OliveGreen}{2}}{\color{NavyBlue}{1}}{\color{BrickRed}{0}}{\color{OliveGreen}{3}}{\color{NavyBlue}{2}}{\color{BrickRed}{4}}{\color{Plum}{1}}{\color{OliveGreen}{0}}{\color{NavyBlue}{3}}{\color{BrickRed}{4}}{\color{Plum}{1}}{\color{OliveGreen}{4}}{\color{NavyBlue}{1}}{\color{BrickRed}{3}}{\color{Plum}{3}}{\color{OliveGreen}{0}}{\color{NavyBlue}{0}}{\color{BrickRed}{2}}{\color{Plum}{4}}{\color{OliveGreen}{4}}{\color{NavyBlue}{2}}{\color{BrickRed}{0}}{\color{Dandelion}{1}}{\color{Plum}{2}}{\color{OliveGreen}{4}}{\color{NavyBlue}{2}}{\color{BrickRed}{2}}{\color{Dandelion}{2}}{\color{Plum}{2}}{\color{OliveGreen}{3}}{\color{NavyBlue}{4}}{\color{BrickRed}{2}}{\color{Dandelion}{4}}{\color{Plum}{0}}{\color{OliveGreen}{3}}{\color{NavyBlue}{1}}{\color{BrickRed}{4}}\ldots
\end{align}

We can imagine the process of concatenating the Fibonacci sequence in base $5$ through this coloring technique by taking the digits $0,1,2,3,4$ in a piecemeal fashion extracted from each sequence $\left(\Phi_{5^k}(n)\right)_{n\in\mathbb{N}}$ and threading them into the appropriate place. Every digit in the Fibonacci concatenation in base $5$ belongs to a particular $\left(\Phi_{5^k}(n)\right)_{n\in\mathbb{N}}$. Further, all of the digits in ${\color{BrickRed}{\left(\Phi_{5^0}(n)\right)_{n\in\mathbb{N}}}}$ occur with a frequency of $1/5$, and all of the digits in ${\color{NavyBlue}{\left(\Phi_{5^1}(n)\right)_{n\in\mathbb{N}}}}$ occur with a frequency of $1/5$, and all of the digits in ${\color{OliveGreen}{\left(\Phi_{5^2}(n)\right)_{n\in\mathbb{N}}}}$ occur with a frequency of $1/5$, etc. That is, the limiting relative frequency of every digit $0,1,2,3,4$ in each place value of the Fibonacci concatenation base $5$ is $1/5$. Then the limiting relative frequency of each digit $0,1,2,3,4$ in the Fibonacci concatenation base $5$ is $1/5$. It follows that the concatenation of the Fibonacci sequence in base $5$ is \textit{simply normal}.
\end{proof}

Note that the result by Niederreiter \ref{lemma:Niederreiter} shows that the Fibonacci sequence is uniformly distributed modulo \textit{every} power of $5$. It immediately follows that the same coloring technique used in Theorem~\ref{thm:coloring} can be used to show that the Fibonacci concatenation is simply normal in every base $\beta=5^x$.
\begin{proposition}\label{thm:simplynormal5x}
In every base of the form $5^x$, the concatenation of the Fibonacci sequence is \textit{simply normal}.
\end{proposition}
\begin{proof}
From Lemma~\ref{lemma:Niederreiter} we know that in every base of the form $5^x$, the digits $0,1,2,\ldots,5^x-1$ in the Fibonacci sequence base $5^x$ are uniformly distributed. By induction on the result of Proposition \ref{prop:base5unif}, the digits in the sequences $\left(\Phi_{{(5^x)}^k}(n)\right)_{n\in\mathbb{N}}$ are uniformly distributed. We can apply the same coloring technique that was used in Theorem \ref{thm:coloring} to every base that is a power of $5$ because the digits are always uniformly distributed. It follows that the Fibonacci concatenation in every base of the form $\beta=5^x$ is \textit{simply normal}. 
\end{proof}

\begin{corollary}
In every base of the form $5^x$, the Fibonacci concatenation is \textit{normal}.
\end{corollary}
\begin{proof}
This follows immediately from Proposition~\ref{thm:simplynormal5x} and Lemma \ref{lemma:Pillai}.
\end{proof}

\section{Bases of the form $2^y$}
Consider the Fibonacci sequence in base 2 where the only digits are $0$ and $1$. We are looking at the sequence:
\[ F_{n,2}= 0, 1, 1, 10, 11, 101, 1000, 1101, 10101, 100010, 110111, 1011001, 10010000, 11101001,\ldots\]

We would like to apply the same coloringing technique we used in base $5$ to the Fibonacci concatenation in base $2$. Unfortunately, the frequency of $0$'s and $1$'s is not equal in the first five place values.
\begin{table}[H]
\centering
\begin{tabular}{c|c|c|c|c}
$k$'s place&$\pi(2^{k+1})$&$\left(\Phi_{2^k}(n)\right)$&$0$'s&$1$'s\\
\hline
$2^0$'s place&3&$0,1,1\circlearrowleft$&$1$&$2$\\
\hline
$2^1$'s place&$6$&$0, 0, 0, 1, 1, 0\circlearrowleft$&$4$&$2$\\
\hline
$2^2$'s place&$12$&$0, 0, 0, 0, 0, 1, 0, 1, 1, 0, 1, 0\circlearrowleft$&$8$&$4$\\
\hline
$2^3$'s place&$24$&$0, 0, 0, 0, 0, 0, 1, 1, 0, 0, 0, 1, 0, 1, 1, 0, 1, 1, 1, 0, 1, 0, 1, 0\circlearrowleft$&$10$&$14$\\
\hline
$2^4$'s place&$48$&$0, 0, 0, 0, 0, 0, 0, 0, 1, 0, 1, 1, 1, 0, 1, 0, 1, 1, 1, 1, 0, 0, 0, 1,$&&\\
&&$0, 1, 1, 0, 1, 1, 0, 1, 0, 0, 0, 0, 1, 1, 0, 0, 0, 0, 1, 0, 1, 0, 1, 0\circlearrowleft$&$28$&$20$
\end{tabular}
\caption{Distribution of $0$'s and $1$'s in the first five place values of the Fibonacci sequence in base $2$.}
\end{table}
However, from Theorem~\ref{Jacobson} below, the sequences $\left(\Phi_{2^k}(n)\right)_{n=1}^{N_k}$ are uniformly distributed for all $k\vargeq5$.
For example, in the $2^5$'s place, there are exactly forty eight $0$'s and forty eight $1$'s. 
\begin{align}
\nonumber \left(\Phi_{2^5}(n)\right)_{n=1}^{96}=&0, 0, 0, 0, 0, 0, 0, 0, 0, 1, 1, 0, 0, 1, 1, 1, 0, 1, 0, 0, 1, 0, 1, 1,\\
\nonumber &1, 0, 1, 0, 1, 1, 1, 0, 0, 1, 1, 0, 1, 1, 1, 1, 0, 1, 1, 1, 0, 0, 0, 1,\\
\nonumber &0, 1, 1, 0, 1, 1, 0, 1, 1, 1, 0, 1, 0, 0, 0, 1, 1, 0, 0, 1, 0, 0, 0, 0,\\
\nonumber &1, 1, 0, 0, 0, 0, 1, 1, 1, 1, 0, 1, 1, 0, 0, 1, 1, 0, 1, 0, 1, 0, 1, 0\circlearrowleft .
\end{align}

The proof technique that we will use to demonstrate normality of the Fibonacci concatenation in bases of the form $2^y$ is very similar to the coloring technique, except that we will ignore the first five place values.
\begin{theorem}\label{thm:LimitingFreq}
For a base $\beta$, if there exists a $K$ such that for all $k\vargeq K$, the digits $0, 1, 2,\ldots, \beta-1$ in the $\beta^k$'s place of the Fibonacci sequence base $\beta$ occur with a frequency of $1/\beta$, then the Fibonacci concatenation in base $\beta$ is \textit{simply normal}.
\end{theorem}
\begin{proof}
Suppose in base $\beta$ that the Fibonacci concatentaion is not uniformly distributed for the first $K$ places, but for all $k\vargeq K$, the digits in the $\beta^k$'s place are uniformly distributed. The digits in the $\beta^0$'s place, $\beta^1$'s place, $\ldots, \beta^K$'s place do not effect the limiting frequency of each digit, in the sense that the digits in these first place values contribute a Lebesgue measure $0$ number of digits to the Fibonacci concatenation base $\beta$. There are a finite number of place values that are not uniformly distributed, and an infinite number of place values that are uniformly distributed. It follows that the measure of the digits in the $\beta^0$'s place, $\beta^1$'s place, $\ldots, \beta^K$'s place is zero. 
\end{proof}

 Another way to say this is that the digits in the $\beta^0$'s place, $\beta^1$'s place, $\ldots, \beta^K$'s place get overwhelmed in the limit. Suppose we choose a digit at random from the Fibonacci concatenation in such a base; the odds of choosing a digit from a place value that is not uniformly distributed has probability $0$. With this theorem in mind, we turn our attention back to the Fibonacci sequence in base $2$.

The distribution of digits of the Fibonacci sequence modulo $2^k$ is well established. Jacobson \cite{[Jacobson]} completely described the number of occurrences of $0$'s and $1$'s. He observed that for the Fibonacci sequence modulo $2^k$, there is a ``type of stability" that occurs when $k\vargeq 5$. 
\begin{theorem}[Jacobson]\label{Jacobson}
Denote the number of occurrences of $z$ as a residue in one (shortest) period of the Fibonacci sequence modulo $m$ by $v(m,z)$. Then for the Fibonacci sequence modulo~$2^k$ and for $k\vargeq5$, we have:
\begin{align*}
	v(2^k,z)=&\begin{cases} 
     						 1 & \text{if}\  z\equiv3\mod4\\
   						    2 & \text{if}\  z\equiv0\mod8\\
  						    3 & \text{if}\  z\equiv1\mod4\\
							 8 & \text{if}\  z\equiv2\mod32\\
							0 & \text{for all other residues.} 
  						\end{cases}
\end{align*}
\end{theorem}
Jacobson recognized that after the $2^4$'s place, there is a regular pattern that emerges. This  stability of the Fibonacci sequence modulo $2$ allows us to make observations about the Fibonacci concatenation in base $2$. 

\begin{theorem}\label{thm:UniDistBase2}
For all $k\vargeq 5$, the digits $0$ and $1$ in $\left(\Phi_{2^k}(n)\right)_{n\in\mathbb{N}}$ are uniformly distributed.
\end{theorem}
\begin{proof}
Consider the Fibonacci sequence in base $2$ and let $k\vargeq5 $. We will represent numbers written in binary modulo $2^{k+1}$ in the following way:
\[
1\underbrace{\ldots\ldots\ldots}_\text{$k-5$ \text{ digits}} a_4a_3a_2a_1a_0 \quad \text{ and } \quad 0\underbrace{\ldots\ldots\ldots}_\text{$k-5$ \text{ digits}}a_4a_3a_2a_1a_0.
\]
Consider the occurence of numbers congruent to $3$ modulo $4$ in the Fibonacci sequence modulo $2^{k+1}$. Every number that has a residue of $3$ modulo $4$ occurs exactly once in one (shortest) period of the Fibonacci sequence modulo $2^{k+1}$. That is, the numbers 
\[
1\underbrace{\ldots\ldots\ldots}_\text{$k-5$ \text{ digits}} a_4a_3a_211 \quad \text{ and } \quad 0\underbrace{\ldots\ldots\ldots}_\text{$k-5$ \text{ digits}}a_4a_3a_211.
\]
each occur exactly once in $\left(\Phi_{2^k}(n)\right)_{n=1}^{N_k}$. Next, consider the occurence of numbers congruent to $0$ modulo $8$ in the Fibonacci sequence modulo $2^{k+1}$. Every number that has a residue of $0$ modulo $8$ occurs exactly twice in one (shortest) period of the Fibonacci sequence modulo $2^{k+1}$. That is, the numbers 
\[
1\underbrace{\ldots\ldots\ldots}_\text{$k-5$ \text{ digits}} a_4a_3000 \quad \text{ and } \quad 0\underbrace{\ldots\ldots\ldots}_\text{$k-5$ \text{ digits}}a_4a_3000.
\]
each occur exactly twice in $\left(\Phi_{2^k}(n)\right)_{n=1}^{N_k}$. Next, consider the occurence of numbers congruent to $1$ modulo $4$ in the Fibonacci sequence modulo $2^{k+1}$. Every number that has a residue of $1$ modulo $4$ occurs exactly three times in one (shortest) period of the Fibonacci sequence modulo $2^{k+1}$. That is, the numbers 
\[
1\underbrace{\ldots\ldots\ldots}_\text{$k-5$ \text{ digits}} a_4a_3a_201 \quad \text{ and } \quad 0\underbrace{\ldots\ldots\ldots}_\text{$k-5$ \text{ digits}}a_4a_3a_201.
\]
each occur exactly three times in $\left(\Phi_{2^k}(n)\right)_{n=1}^{N_k}$. Finally, consider the occurence of numbers congruent to $2$ modulo $32$ in the Fibonacci sequence modulo $2^{k+1}$. Every number that has a residue of $2$ modulo $32$ occurs exactly eight times in one (shortest) period of the Fibonacci sequence modulo $2^{k+1}$. That is, the numbers 
\[
1\underbrace{\ldots\ldots\ldots}_\text{$k-5$ \text{ digits}}00010 \quad \text{ and } \quad 0\underbrace{\ldots\ldots\ldots}_\text{$k-5$ \text{ digits}}00010.
\]
each occur exactly eight times in $\left(\Phi_{2^k}(n)\right)_{n=1}^{N_k}$. All other residues occur exactly zero times in the Fibonacci sequence modulo $2^{k+1}$ for $k\vargeq 5$. Then there is a one-to-one correspondence between numbers in the smallest period with a 1 in the $2^k$ digit and numbers with a 0 in the $2^k$ digit. Hence, the digits in the Fibonacci sequence are uniformly distributed in every $2^k$'s place for $k\vargeq 5$. 
\end{proof}
Now that we know that the digits $0$ and $1$ are uniformly distributed for sufficiently large place values, we have the following result:

\begin{corollary}\label{SimplyNormalPowers2}
The Fibonacci concatenation is \textit{simply normal} in every base of the form $\beta=2^k$.
\end{corollary}

\begin{proof}
This immediately follows from Theorem~\ref{thm:LimitingFreq} and Theorem ~\ref{thm:UniDistBase2}.
\end{proof}
Now that we know every base that is a power of $2$ is simply normal, we can invoke Borel's definition of normality and achieve a stronger result:
\begin{theorem}
The Fibonacci concatenation is \textit{normal} in every base of the form $\beta=2^y$.
\end{theorem}

\begin{proof}
This follows immediately from Corollary~\ref{SimplyNormalPowers2} and Lemma~\ref{lemma:Pillai}
\end{proof}

\section{Bases of the form $5^x2^y$}\label{sec:MainThm}

Consider now bases of the form $5^x2^y$. Because $5$ and $2$ are not Wall-Sun-Sun primes, by Lemma~\ref{lemma:WallCriteria} we know that $\pi(5^x)=20(5^{x-1})$ and $\pi(2^y)=3(2^{y-1})$. Further, because $2^y$ and $5^x$ will always be coprime, by Lemma \ref{lemma:lcm} we know that 
\begin{align*}
\pi(5^x2^y)=\text{LCM}\left(\pi(5^x),\pi(2^y)\right)=\begin{cases} 
     																	 12\times5^x & \text{if}\  1\varleq y\varleq3\\
   																	 12\times5^x\times2^{y-3} & \text{if}\ y>3.
  																\end{cases}
\end{align*}
For example, if $x=y=1$, then $\pi(2\times5)=12(5^1)=60=\pi(10)$. Jacobson \cite{[Jacobson]} considered the Fibonacci sequence modulo integers of the form $5^x2^y$ and obtained the following result:

\begin{theorem}[Jacobson]\label{Jacobson2}
Denote the number of occurrences of $z$ as a residue in one (shortest) period of the Fibonacci sequence modulo $5^x2^y$ by $v(5^x2^y,z)$. Then for $x\vargeq0$ and $y\vargeq5$ we have that 
\[
v(5^x2^y,z)=\begin{cases} 
     						 1 & \text{if}\  z\equiv3\mod4\\
   						    2 & \text{if}\  z\equiv0\mod8\\
  						    3 & \text{if}\  z\equiv1\mod4\\
							 8 & \text{if}\  z\equiv2\mod32\\
							0 & \text{for all other residues.} 
  						\end{cases}
\]
\end{theorem}
Jacobson showed here that because the Fibonacci sequence is uniformly distributed modulo every power of $5$, the residues in the Fibonacci sequence modulo any power of $5^x2^y$ will be the same as the residues modulo any power of $2^y$ provided, $y\vargeq5$.

\begin{theorem}\label{thm:UniDistBase5x2y}
For $\beta=5^x2^y$, and for all $k\vargeq 5$, the digits $0, 1, \ldots, \beta-1$ in $\left(\Phi_{\beta^k}(n)\right)_{n\in\mathbb{N}}$ are uniformly distributed.
\end{theorem}

\begin{proof}
Let $\beta = 5^x 2^y$. The case $x=0$ is exactly the statement of Theorem~\ref{Jacobson}, and the case $y=0$ is exactly the statement of Proposition~\ref{prop:base5unif}. 
Assume now that $x\geq1$ and $y\geq1$, so in particular $10 \mid \beta$. The  proof  follows the same structure as Theorem~\ref{Jacobson}, but the bookkeeping is a bit more complicated. 

Since $10 \mid \beta$, we have $4 \mid 100_\beta$, $8 \mid 1000_\beta$, and $32 \mid 100000_\beta$, meaning a number is 3 (resp.~1) modulo 4 exactly when the last two digits are 3 (resp.~1)  modulo 4, a number is 2 modulo 8 exactly when the last three digits are 2 modulo 8, and a number is 8 modulo 32 exactly when the last five digits are 8 modulo 32. 

Let $a_1a_0$ be a two-digit number in base $\beta$ that is congruent to 3 modulo 4. Then from~\ref{Jacobson2}, each of these numbers (written in base $\beta$) appears exactly once in one (shortest) period of the Fibonacci sequence modulo $\beta^{k+1}$:
\begin{align*}
\beta-1&\underbrace{\ldots\ldots\ldots}_\text{$k-5$ \text{ digits}} a_4a_3a_2a_1a_0\\
&\vdots\\
1&\underbrace{\ldots\ldots\ldots}_\text{$k-5$ \text{ digits}} a_4a_3a_2a_1a_0\\
0&\underbrace{\ldots\ldots\ldots}_\text{$k-5$ \text{ digits}}a_4a_3a_2a_1a_0.
\end{align*}

Similarly, if $a_1a_0$ is a two-digit number in base $\beta$ that is congruent to 1 modulo 4, then each of these numbers (written in base $\beta$) appears exactly three times in one (shortest) period of the Fibonacci sequence modulo~$\beta^{k+1}$:
\begin{align*}
\beta-1&\underbrace{\ldots\ldots\ldots}_\text{$k-5$ \text{ digits}} a_4a_3a_2a_1a_0\\
&\vdots\\
1&\underbrace{\ldots\ldots\ldots}_\text{$k-5$ \text{ digits}} a_4a_3a_2a_1a_0\\
0&\underbrace{\ldots\ldots\ldots}_\text{$k-5$ \text{ digits}}a_4a_3a_2a_1a_0.
\end{align*}

If $a_2a_1a_0$ represents a three-digit number in base $\beta$ that is congruent to 0 modulo 8, then each of these numbers appears exactly twice in one (shortest) period of the Fibonacci sequence modulo~$\beta^{k+1}$:
\begin{align*}
\beta-1&\underbrace{\ldots\ldots\ldots}_\text{$k-5$ \text{ digits}} a_4a_3a_2a_1a_0\\
&\vdots\\
1&\underbrace{\ldots\ldots\ldots}_\text{$k-5$ \text{ digits}} a_4a_3a_2a_1a_0\\
0&\underbrace{\ldots\ldots\ldots}_\text{$k-5$ \text{ digits}}a_4a_3a_2a_1a_0.
\end{align*}

And if $a_4a_3a_2a_1a_0$ represents a five-digit number that is congruent to 2 modulo 32, then each of these numbers appears exactly eight times in one (shortest) period of the Fibonacci sequence modulo $\beta^{k+1}$:
\begin{align*}
\beta-1&\underbrace{\ldots\ldots\ldots}_\text{$k-5$ \text{ digits}} a_4a_3a_2a_1a_0\\
&\vdots\\
1&\underbrace{\ldots\ldots\ldots}_\text{$k-5$ \text{ digits}} a_4a_3a_2a_1a_0\\
0&\underbrace{\ldots\ldots\ldots}_\text{$k-5$ \text{ digits}}a_4a_3a_2a_1a_0.
\end{align*}

No other numbers appear in the Fibonacci sequence modulo $\beta^{k+1}$. This one-to-one correspondence shows that each first digit $0, 1, \dots, \beta-1$ appears exactly as often as every other digit in the Fibonacci sequence modulo $\beta^{k+1}$. It follows that the digits $0, 1, \dots, \beta-1$ in $\left(\Phi_{\beta^k}(n)\right)_{n\in\mathbb{N}}$ are uniformly distributed for $k\vargeq 5$.
\end{proof}

\begin{proposition}\label{SimplyNormalPowers5x2y}
The Fibonacci concatenation is \textit{simply normal} in every base of the form $\beta=5^x2^y$.
\end{proposition}
\begin{proof}
This immediately follows from Theorem~\ref{thm:UniDistBase5x2y} and Theorem~\ref{thm:UniDistBase5x2y}.
\end{proof}
As before, now that we know every base that is a power of $5^x2^y$ is simply normal, we can invoke Borel's definition of normality and achieve a stronger result.

\begin{proof}[Proof of Main Theorem]
The proof that the Fibonacci concatenation is \textit{normal} in every base of the form $\beta=5^x2^y$ immediately follows from Proposition~\ref{SimplyNormalPowers5x2y} and Lemma~\ref{lemma:Pillai}.
\end{proof}

\section{Other bases}
In this section we provide computational evidence to suggest that our Main Theorem holds for infinitely many bases, perhaps in every base. We begin with bases $\beta$ such that $\omega(\beta)=4$~\cite{[A053029]}:
\[
5, 10, 13, 17, 25, 26, 34, 37, 50, 53, 61, 65, 73, 74, 85, 89, 97, 106, 109, 113, 122, 125, 130, 137, 146, 149,\ldots
\]
We have seen that in the smallest entry on this list, base $5$, the Fibonacci concatenation is normal. There is computational and heuristic evidence to believe the Fibonacci concatenation is normal in \textit{every} base on this list. We begin with a definition.

\begin{definition}\label{def:Upsilon}
For an integer base $\beta$, let $\Upsilon(\beta)$ denote the smallest $K$ such that for all $k\vargeq K$, the digits $0,1,2,\ldots,\beta-1$ are uniformly distributed in $\left(\Phi_{\beta^k}(n)\right)_{n=1}^{N_k}$. 
\end{definition}
\begin{conjecture}\label{conj:omega4UD}
For every base $\beta$ where $\omega(\beta)=4$, there exists a $K$ such that for all $k\vargeq K$, the digits in the sequences $\left(\Phi_{\beta^k}(n)\right)_{n\in\mathbb{N}}$ are uniformly distributed.
\end{conjecture}
In fact, the data collected in Table 6\footnote{Computations done in Mathematica.} below suggests that for all \textit{prime} bases $p>5$ where $\omega(p)=4$, conjecture \ref{conj:omega4UD} holds for $K=1$. In other words, the only time $\left(\Phi_{p^k}(n)\right)_{n\in\mathbb{N}}$ is not uniformly distributed is in the ones place. 
\begin{table}[H]\label{table6}
\begin{tabular}{c|c|c}
Base $\beta$&$\Upsilon(\beta)$&Searched to $\beta^k$'s place\\
\hline
5&0&Proven normal\\
\hline
13&1&$13^4$'s place\\
\hline
17&1&$17^4$'s place\\
\hline
37&1&$37^3$'s place\\
\hline
53&1&$53^2$'s place\\
\hline
61&1&$61^2$'s place\\
\end{tabular}
\caption{The first few prime bases with four zeros in their Pisano period.}
\end{table}

\begin{conjecture}\label{conj:omega4normal}
In every base $\beta$ where $\omega(\beta)=4$, and for all nonnegative $x$, the Fibonacci concatenation is normal in every base of the form $2^y\beta$.
\end{conjecture} 
First, computations below for small $\beta$ on the list indicate that the digits in each place beyond $\beta^0$ are uniformly distributed. If Conjecture \ref{conj:oeis} holds, then whenever $\omega(\beta)=4$ then $\omega(\beta^i)=4$ for all $i\vargeq1$. This suggests
 applying Theorem~\ref{thm:LimitingFreq} to show that the Fibonacci concatenation is simply normal in these bases. Furthermore, from Theorem~\ref{thm:Renault}, if $\beta_1$ and $\beta_2$ are relatively prime such that $\omega(\beta_1) = \omega(\beta_2) = 4$, then $\omega(\beta_1\beta_2) = 4$ as well. Similarly, if $\omega(\beta)=4$, then $\omega(2^y\beta)=4$. All of this suggests that one could mimic the proofs in Section~\ref{sec:MainThm} in these cases.
We could then determine what happens when a base has a prime factorization $\beta=2^yp_1^{e_1}p_2^{e_2}\cdots p_j^{e_j}$ where for every odd prime $p$, $\omega(p)=4$.

Jacobson \cite{[Jacobson]} demonstrated how to combine information about distribution of Fibonacci numbers modulo relatively prime bases $5^x$ and $2^y$ into a result modulo $5^x2^y$. We suggest a generalization of his approach phrased in terms of uniform distribution of digits:
\begin{conjecture}\label{UpsilonMAX}
Given two coprime bases $\beta_1$ and $\beta_2$ such that $\Upsilon(\beta_1)=u$ and $\Upsilon(\beta_2)=v$ then $\Upsilon(\beta_1\beta_2)=\text{max}(uv)$.
Inductively, it would follow for any finite set of coprime bases $\beta_1, \beta_2, \ldots, \beta_n$ where $\Upsilon(\beta_1)=u_1$, $\Upsilon(\beta_2)=u_2$, $\ldots$, $\Upsilon(\beta_n)=u_n$ then $\Upsilon(\beta_1, \beta_2, \ldots, \beta_n)=\text{max}(u_1,u_2,\ldots,u_n)$. 

\end{conjecture}

If we are able to determine that the sequences $\left(\Phi_{\beta^k}(n)\right)_{n\in\mathbb{N}}$ are uniformly distributed for each base of the form $2^y\beta$ where $\omega(\beta)=4$ and $y\vargeq0$, we could proceed as above. Applying Theorem \ref{thm:LimitingFreq} we know that every base of the form $\beta\times2^y$ (including powers of these bases from Conjecture~\ref{conj:oeis}) is \textit{simply normal}, and by Lemma~\ref{lemma:Pillai} we would conclude that every base where $\omega(\beta)=4$ is \textit{normal}.

We now turn our attention to bases not yet covered by Conjecture~\ref{conj:omega4normal}. For a base $\gamma$ that falls outside the criteria of Conjecture \ref{conj:omega4normal},  the sequences $\left(\Phi_{\gamma^k}(n)\right)_{n\in\mathbb{N}}$ do not obviously become uniformly dirstibuted based on our computer search. It is possible that there is a $K$ such that for all $k\vargeq K$ the sequences $\left(\Phi_{\gamma^k}(n)\right)_{n\in\mathbb{N}}$ are all uniformly distributed, but if this is the case, the smallest such $K$ is beyond our computational power. Nevertheless,  we provide below computational evidence to believe that the Fibonacci concatenation is \textit{normal} in every base.
\begin{conjecture}\label{conj:absolutelynormal}
The Fibonacci concatenation is \textit{absolutely normal}.
\end{conjecture}
We will build the argument to support Conjecture \ref{conj:absolutelynormal} like before, by showing that the Fibonacci concatenation is \textit{simply normal} in every base. It follows that the Fibonacci concatenation is simply normal in every \emph{power} of every base. Then, by Lemma~\ref{lemma:Pillai}, the Fibonacci concatenation is \textit{normal} in every base.

\begin{conjecture}\label{conj:gammasimplynormal}
The Fibonacci concatenation is simply normal in \textit{every} base.
\end{conjecture}

\begin{example*}
Consider again the Fibonacci sequence in base 3. Below is a table of the first few sequences of $\left(\Phi_{3^k}(n)\right)_{n}^{N_k}$ along with a running total of the frequency of the digits $0, 1, 2$ and a running percentage showing the distribution at finite levels. The running total and running percentages are calculated using the nesting property (Lemma \ref{nesting}): in one period of $3^k$'s place digits there are three periods of $3^{k-1}$'s place digits, nine periods of $3^{k-2}$'s place digits, twenty seven periods of $3^{k-3}$'s place digits etc.
\begin{table}[H]
\centering
\begin{tabular}{c|c|c|c|c|c}
$3^k$'s&$0$'s&$1$'s&$2$'s&Running total of $0$'s:$1$'s:$2$'s&Running Percentage of $0$'s:$1$'s:$2$'s\\
\hline
$3^0$'s&2&3&3&2 : 3 : 3&25.0000\% : 37.5000\% : 37.5000\%\\
\hline
$3^1$'s&9&6&9&15 : 15 : 18&31.2500\% : 31.2500\% : 37.5000\%\\
\hline
$3^2$'s&27&18&27&72 : 63 : 81&$33.\overline{3333}$\% : $29.1\overline{666}$\% : 37.5000\%\\
\hline
$3^3$'s&75&66&75&291 : 255 : 318&$33.680\overline{5}\%$ : $29.513\overline{8}$\% : $36.80\overline{55}\%$\\
\hline
$3^4$'s&216&216&216&1089 : 981 : 1170&$33.6\overline{111}$\% : $30.2\overline{777}$\% : $36.\overline{1111}$\%\\
\hline
$3^5$'s&630&684&630&3897 : 3627 : 4140&33.4105\% : 31.0957\% : 35.4938\%\\
\hline
$3^6$'s&1971&1890&1971&13662 : 12771 : 14391&33.4656\% : 31.2831\% : 35.2513\%\\
\hline
$3^7$'s&5859&5778&5859&46845 : 44091 : 49032&33.4684\% : 31.5008\% : 35.0309\%\\
\hline
$3^8$'s&17577&17334&17577&158112 : 149607 : 164673&33.4705\% : 31.6701\% : 34.8594\%\\
\hline
$3^9$'s&52326&52812&52326&5266621 : 501633 : 546345&33.4465\% : 31.8570\% : 34.6964\%\\
\hline
$3^{10}$'s&157707&156978&157707&1737693 : 1661877 : 1796742&33.4409\% : 31.9818\% : 34.4577\%\\
\hline
$3^{11}$'s&472635&471906&472635&5685714 : 5457537 : 5862861&33.4334\% : 32.0916\% : 34.4750\%
\end{tabular}
\caption{Table of the number of times each digit appears in $\left(\Phi_{3^k}(n)\right)_{n=1}^{N_k}$ for $k\varleq11$.}
\end{table}
\end{example*}

Notice in the table that the limiting frequency of each digit seems to approach $1/3$. If this running percentage of the number of times each digit $0, 1, 2$ appears in  $\left(\Phi_{3^k}(n)\right)_{n=1}^{N_k}$ does in fact tend to $1/3$, then each digit will appear equally often in the limit of the Fibonacci concatenation. This is exactly what we would expect from a number that is \textit{simply normal}. From this, we could stitch together the Fibonacci concatenation in base $3$ exactly like we did before with the coloring technique. Note that base $3$ is not the only base $\gamma$ in which the running percentage of digits appears to converge to $1/\gamma$. In every base we have studied, the running percentage of digits in base $\gamma$  converges to $1/\gamma$. Below, we graph the running percentage of the first few such bases. 

\begin{figure}[H]
\minipage{0.32\textwidth}
\caption*{Base 3}
\includegraphics[width=\linewidth]{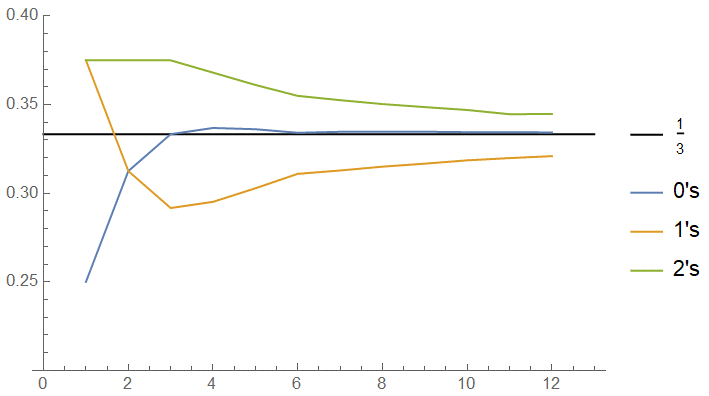}
\caption*{$3^0$'s$\ldots3^{11}$'s}
\endminipage\hfill
\minipage{0.32\textwidth}
\caption*{Base 6}
\includegraphics[width=\linewidth]{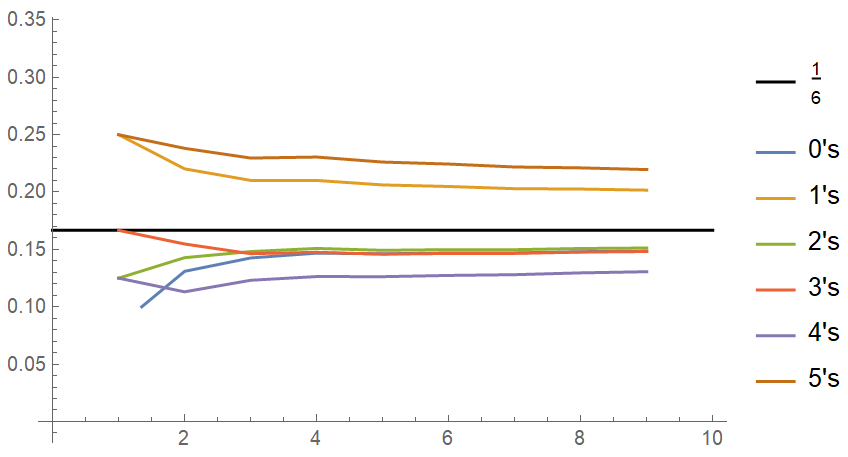}
\caption*{$6^0$'s$\ldots6^{8}$'s}
\endminipage\hfill
\minipage{0.32\textwidth}
\caption*{Base 7}
\includegraphics[width=\linewidth]{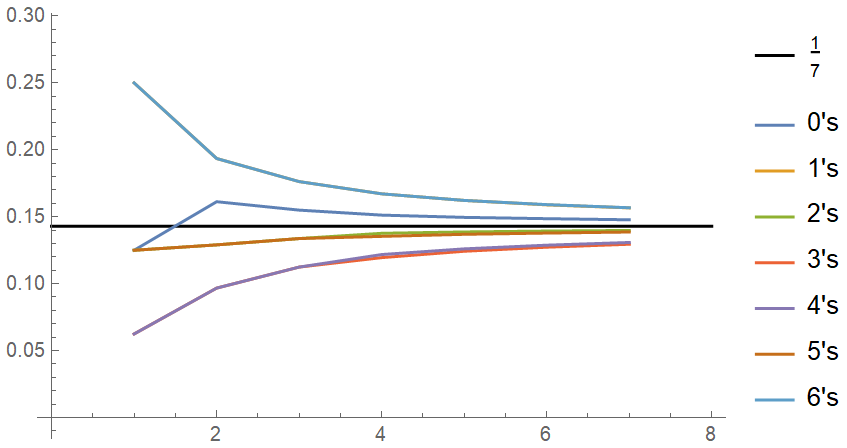}
\caption*{$7^0$'s$\ldots7^{6}$'s}
\endminipage\hfill
\minipage{0.32\textwidth}
\caption*{Base 9}
\includegraphics[width=\linewidth]{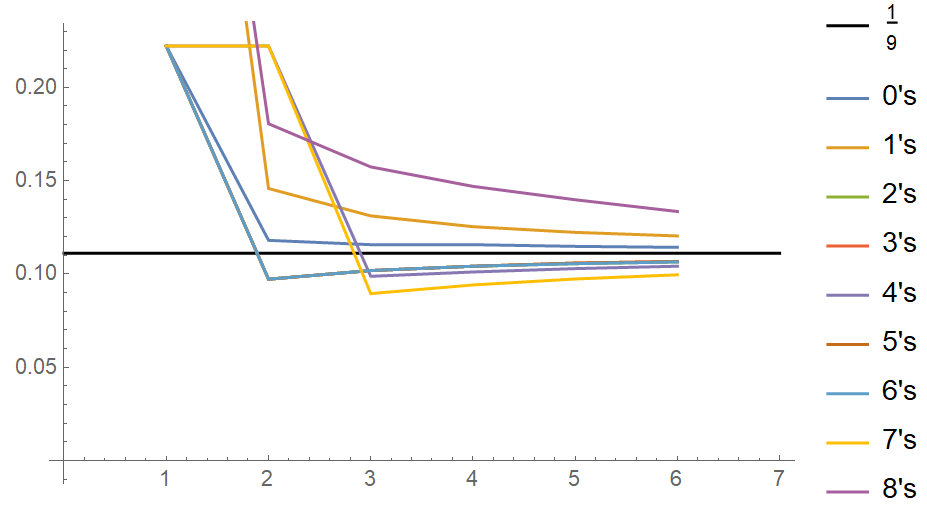}
\caption*{$9^0$'s$\ldots9^{5}$'s}
\endminipage\hfill
\minipage{0.32\textwidth}
\caption*{Base 11}
\includegraphics[width=\linewidth]{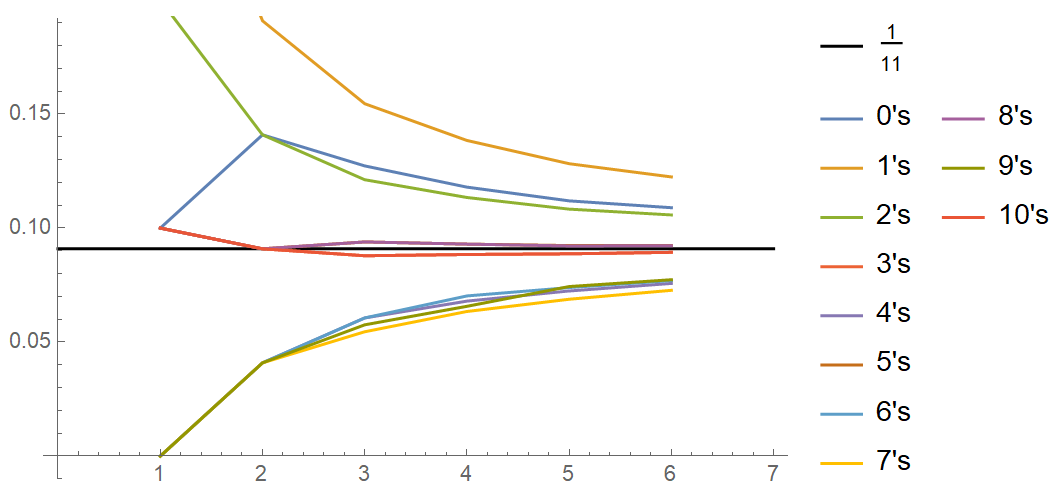}
\caption*{$11^0$'s$\ldots11^{5}$'s}
\endminipage\hfill
\minipage{0.32\textwidth}
\caption*{Base 12}
\includegraphics[width=\linewidth]{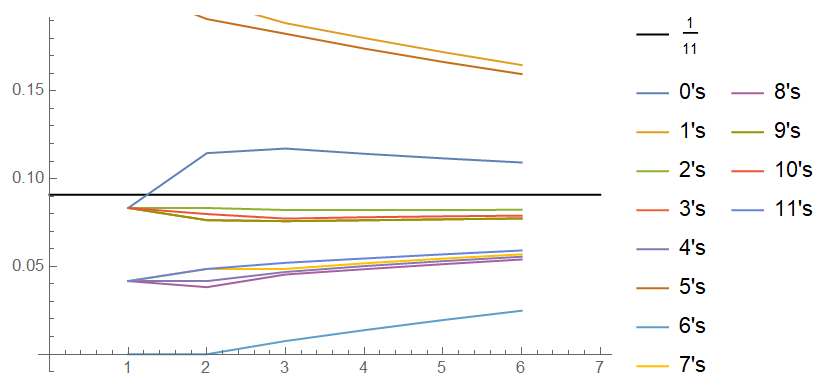}
\caption*{$12^0$'s$\ldots12^{5}$'s}
\endminipage\hfill
\minipage{0.32\textwidth}
\caption*{Base 14}
\includegraphics[width=\linewidth]{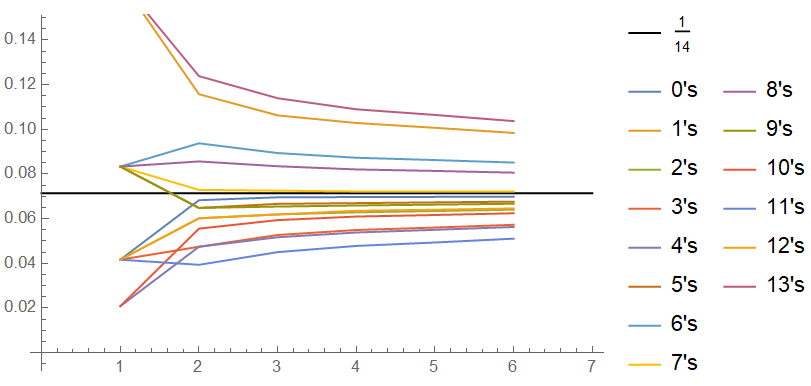}
\caption*{$14^0$'s$\ldots14^{5}$'s}
\endminipage\hfill
\minipage{0.32\textwidth}
\caption*{Base 15}
\includegraphics[width=\linewidth]{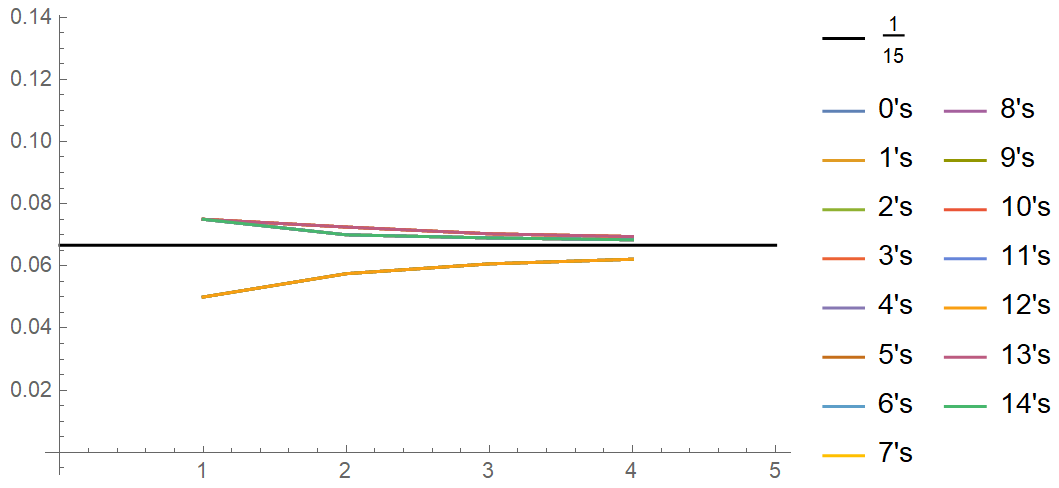}
\caption*{$15^0$'s$\ldots15^{4}$'s}
\endminipage\hfill
\minipage{0.32\textwidth}
\caption*{Base 18}
\includegraphics[width=\linewidth]{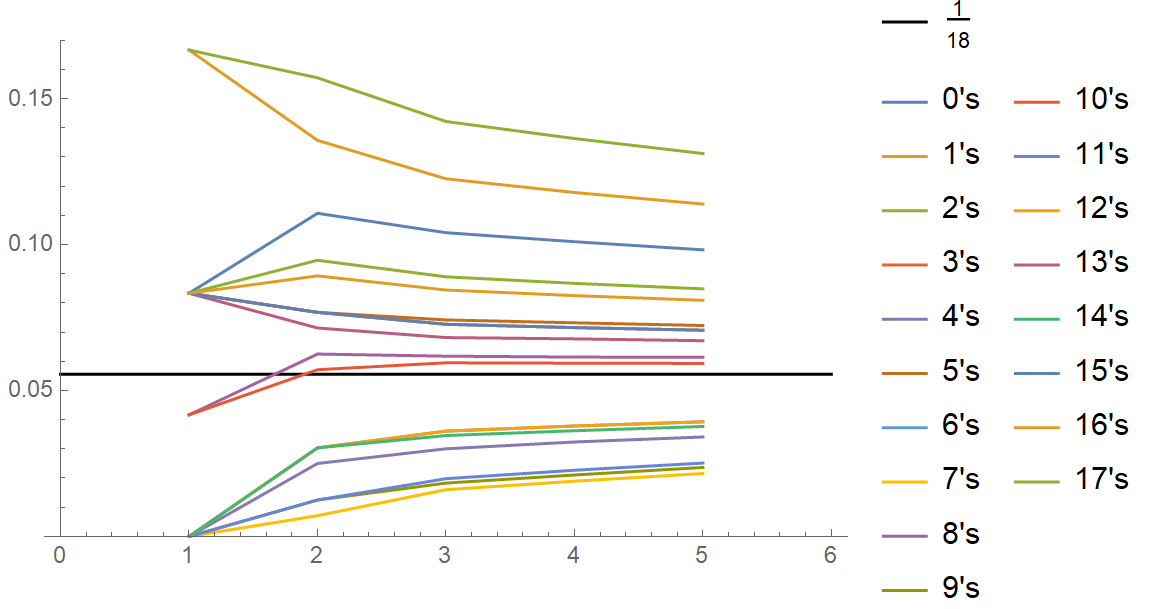}
\caption*{$18^0$'s$\ldots18^{4}$'s}
\endminipage\hfill
\caption{Running Percentage of digits $0, 1, 2, \ldots, 1-\gamma$ with line $1/\gamma$}
\end{figure}
These computations provide evidence that Conjecture \ref{conj:gammasimplynormal} holds.  If Conjecture \ref{conj:gammasimplynormal} holds for a base $\gamma$, then it necessarily holds for every base $\gamma^i$ for all $i$. We could then invoke Lemma \ref{lemma:Pillai} one more time and claim that the Fibonacce concatenation is \textit{absolutely normal}.

\section*{Acknowledgements}
The authors would like to extend their sincerest gratitude to Will Brian for his early encouragement on this project, to Marc Renault, and Joseph Vandehey for their correspondence and expertise, and to Michael De Vlieger for his brilliant Mathematica code.

\bibliographystyle{plain}
\bibliography{FibConcatBib}{}

\end{document}